\documentclass[a4paper,11pt]{amsart}
\pdfoutput=1
\usepackage[utf8]{inputenc}

\usepackage{hyperref}

\usepackage[sc]{mathpazo}
\usepackage[tracking]{microtype}

\usepackage[T1]{fontenc}
\usepackage{textcomp}

\usepackage[arrow,matrix]{xy}
\usepackage{enumerate}

\theoremstyle{plain}
\newtheorem{thm}{Theorem}[section]
\newtheorem{prop}[thm]{Proposition}
\newtheorem{lem}[thm]{Lemma}
\newtheorem{cor}[thm]{Corollary}

\theoremstyle{definition}
\newtheorem{question}[thm]{Question}

\theoremstyle{remark}
\newtheorem{rem}[thm]{Remark}

\DeclareMathOperator{\Aut}{Aut}
\DeclareMathOperator{\End}{End}

\DeclareMathOperator{\id}{id}
\DeclareMathOperator{\Pic}{Pic}

\newcommand{\bC}{\mathbb{C}}
\newcommand{\bF}{\mathbb{F}}
\newcommand{\bN}{\mathbb{N}}
\newcommand{\bP}{\mathbb{P}}
\newcommand{\bQ}{\mathbb{Q}}
\newcommand{\bZ}{\mathbb{Z}}

\newcommand{\sE}{\mathcal{E}}
\newcommand{\sH}{\mathcal{H}}
\newcommand{\sI}{\mathcal{I}}
\newcommand{\sL}{\mathcal{L}}
\newcommand{\sO}{\mathcal{O}}
\newcommand{\sW}{\mathcal{W}}
\newcommand{\sX}{\mathcal{X}}
\newcommand{\sY}{\mathcal{Y}}
\newcommand{\sZ}{\mathcal{Z}}

\begin{document}

\title[Deformations of almost homogeneous projective bundles]{Global deformations of certain rational almost homogeneous projective bundles}

\author{Florian Schrack}
\address{Mathematisches Institut\\ Universität Bayreuth\\ 95440 Bayreuth\\ Germany}
\email{florian.schrack@uni-bayreuth.de}



\begin{abstract}
  We study global deformations of certain projective bundles over projective spaces. We show that any projective global deformation of a projective bundle over~$\bP^1$ carries the structure of a projective bundle over some projective space. Furthermore, we construct examples in arbitrary dimension $\ge 3$ of almost homogeneous Fano projective bundles over~$\bP^2$ which can be globally deformed to non-almost homogeneous manifolds.
\end{abstract}

\maketitle
\section{Introduction}
We let $\Delta \subset \bC$ be the unit disk and consider families $(X_t)_{t\in\Delta}$ of compact complex manifolds, i.e., proper submersions
$p\colon \sX \to \Delta$, where $\sX$ is a complex manifold and $X_t := p^{-1}(t)$ for $t \in \Delta$. In this setup, we call $X_0$ a \emph{(global) deformation} of~$X_t$.

A classical problem is to study the case where for all $t \ne 0$, $X_t$ is isomorphic to some rational homogeneous manifold $S = G/P$, where $G$ is a semisimple Lie group and $P < G$ a parabolic subgroup. The question then is whether also $X_0$ is rational homogeneous (and thus $X_0 \cong S$ by local rigidity of rational homogeneous manifolds, cf.~\cite{Bot57}).

If one assumes $b_2(S) = 1$, Hwang and Mok showed in a series of papers (see~\cite[§3.3]{Mok16} for references and an outline of the proof) that this is true if $S$ is not isomorphic to the 7-dimensional Fano contact manifold $\bF^5$, which has been shown by Pasquier--Perrin in~\cite{PP10} to admit a deformation to a non-homogeneous horospherical variety.

If $b_2(S) > 1$, it is obviously no longer true that $X_0 \cong S$, as can already be seen for $S = \bP^1 \times \bP^1$, which can be deformed to the Hirzebruch surface $\bF_k$ for arbitrary even~$k$.

So we see that $X_0$ need not necessarily be homogeneous. In the examples cited above, however, $X_0$ is still \emph{almost} homogeneous, i.e., $\Aut X_0$ acts on~$X_0$ with an open orbit (this is equivalent to $T_{X_0}$ being \emph{generically} globally generated). We can therefore ask:
\begin{question}\label{qu:homdef}
  Let $(X_t)_{t\in\Delta}$ be a family of compact complex manifolds where $X_t$ is rational homogeneous for $t \ne 0$. Is then $X_0$ almost homogeneous?
\end{question}
We might even ask the following stronger question:
\begin{question}\label{qu:almhomdef}
  Let $(X_t)_{t\in\Delta}$ be a family of compact complex manifolds where $X_t$ is an almost homogeneous manifold for every $t \ne 0$. Is then $X_0$ almost homogeneous?
\end{question}

The purpose of the present article is to give negative answers to both questions: We show that the homogeneous variety~$\bP(T_{\bP^2})$ can be deformed to a non-almost homogeneous manifold. This answers Question~\ref{qu:homdef} negatively for $\dim X_t = 3$ and $b_2(X_t) = 2$. Taking products with projective spaces, one also obtains negative answers for $\dim X_t > 3$ and $b_2(X_t) \ge 3$.

Moreover, for any $n \ge 3$, we construct almost homogeneous $n$-di\-men\-sion\-al projective bundles over~$\bP^2$ which are Fano and can be deformed to non-almost homogeneous manifolds.  These examples give negative answers to Question~\ref{qu:almhomdef} for $\dim X_t \ge 3$ and $b_2(X_t) = 2$.

In order to give some context to the above-mentioned examples, we begin our discussion by reviewing the theory of almost homogeneous compact complex surfaces in the context of global deformations in section~\ref{sec:surf}. We furthermore study projective global deformations of projective bundles over~$\bP^1$ in section~\ref{sec:prbundp1}, where we show that such global deformations again carry the structure of a projective bundle and in particular are almost homogeneous in most cases. This is inspired by work of Brieskorn \cite{Bri65}.

The remanining sections are then devoted to the construction of the examples announced above: In section~\ref{sec:almhomprojbun}, we give some criteria for almost homogeneity of projective bundles. Finally, in section~\ref{sec:degen}, we give the construction of the above-mentioned (almost) homogeneous Fano projective bundles over~$\bP^2$ which admit global deformations to non-almost homogeneous manifolds.

Throughout the article, complex manifolds are assumed to be \emph{connected}. We use the notation~``$\bP(E)$'' to denote the projective bundle of \emph{hyperplanes} in the fibers of a given vector bundle~$E$.

\subsection*{Acknowledgments}
I would like to thank Thomas Peternell for bringing to my attention the questions which inspired this article and for numerous useful discussions.

\section{The surface case}\label{sec:surf}
Almost homogeneous compact complex surfaces have been classified by Potters:
\begin{thm}[\cite{Pot69}]\label{thm:potters}
  Let $S$ be a compact complex surface. If $S$ is almost homogeneous, then one of the following holds:
\begin{enumerate}[(i)]
\item
  $S$ is obtained from $\bP^2$ or from a Hirzebruch surface by blowing up a finite number of points,
\item\label{it:ell}
  $S \cong \bP(V)$, where $V$ is a rank-$2$ bundle over an elliptic curve which is either a direct sum of $\sO$ and a topologically trivial line bundle or a non-split extension of $\sO$ with $\sO$,
\item\label{it:hopf}
  $S$ is a Hopf surface with abelian fundamental group,
\item\label{it:tor}
  $S$ is a two-dimensional complex torus.
\end{enumerate}
Conversely, $\bP^2$, the Hirzebruch surfaces and the surfaces of type (\ref{it:ell}), (\ref{it:hopf}), (\ref{it:tor}) are almost homogeneous.
\end{thm}

It is classically known that any deformation of of a Hirzebruch surface is again a Hirzebruch surface (cf.~\cite{Bri65}). We will investigate projective bundles over~$\bP^1$ in arbitrary dimension in section~\ref{sec:prbundp1}. Global deformations of such bundles will be characterized in Theorem~\ref{thm:prbundp1}.

For ruled surfaces over elliptic curves, global deformations fail to be almost homogeneous in general, as the following example shows (cf.~\cite{PS14}): Let $\sE$ be a rank-$2$ vector bundle over $\bP^1 \times \Delta$ such that $E_t := \sE|_{\bP^1\times\{t\}} \cong \sO_{\bP^1}^{\oplus 2}$ for $t \ne 0$ and $E_0 \cong \sO_{\bP^1}(-k) \oplus \sO_{\bP^1}(k)$ for some $k > 0$. Let $\eta\colon A \to \bP^1$ be a 2-sheeted cover from an elliptic curve~$A$, and let $\sX := \bP((\eta\times\id)^*\sE)$. Then $\sX$ is a family of compact complex surfaces over~$\Delta$ with $X_t \cong \bP^1 \times A$ for $t\ne 0$ and $X_0$ is a ruled surface over~$A$ which is not almost homogeneous by Theorem~\ref{thm:potters}.

For the other cases in Theorem~\ref{thm:potters}, we cite the following classical results by Kodaira and Kodaira--Spencer:
\begin{prop}[{\cite[Thm.~36]{Kod66}}]
  Let $S$ be a Hopf surface. Then any deformation of $S$ is also a Hopf surface.
\end{prop}
\begin{prop}[{\cite[Thm.~20.2]{KS58}}]\label{prop:deftorus}
  Let $S$ be a two-dimensional complex torus. Then any deformation of $S$ is also a complex torus.
\end{prop}
\begin{rem}
  Catanese showed in \cite[Thm.~2.1]{Cat04} that Proposition~\ref{prop:deftorus} is true in any dimension.
\end{rem}

\section{Projective bundles over~$\bP^1$}\label{sec:prbundp1}
In this section we investigate families whose general fiber is isomorphic to a $\bP^r$-bundle over~$\bP^1$ (see Corollary~\ref{cor:prbundp1} for a proof that any such bundle is almost homogeneous).

We first study extremal contractions of projective bundles over~$\bP^1$:
\begin{prop}\label{prop:prbundcontr}
  For given natural numbers $0 \le a_1 \le \dotsb \le a_{r}$ consider the vector bundle $E := \sO_{\bP^1} \oplus \sO_{\bP^1}(a_1) \oplus \dotsb \oplus \sO_{\bP^1}(a_r)$ over $\bP^1$ and let $\pi\colon X := \bP(E) \to \bP^1$ be the associated $\bP^r$-bundle over~$\bP^1$. Let $\varphi\colon X \to Y$ be the contraction of a $K_X$-negative extremal ray of~$\overline{NE}(X)$. Assume $\varphi \ne \pi$. Then either $E \cong \sO_{\bP^1}^{\oplus (r+1)}$ or $E \cong \sO_{\bP^1}^{\oplus r}\oplus \sO_{\bP^1}(1)$.
\end{prop}
\begin{proof}
  First observe that $\mathrm{Nef}(X)$ is spanned as a convex cone by the classes of the tautological bundle $\sO_X(1)$ and the pullback $\pi^*\sO_{\bP^1}(1)$. But now $\mathrm{Nef}(X)$ and $\overline{NE}(X)$ are dual cones, so
  \[
    \overline{NE}(X) = \mathbb{R}^+_0\ell + \mathbb{R}^+_0C,
  \]
  where $\ell$ is the class of a line in a fiber of~$\pi$ and $C$ is some curve class with $C.\sO_X(1) = 0$ and $C.\pi^*\sO_{\bP^1}(1) > 0$.
  From
  \[
    K_X = \sO_X(-r-1) \otimes \pi^*\sO_{\bP^1}(a_1 + \dotsb + a_r - 2)
  \]
  it then follows that $C.K_X < 0$ if and only if $a_1 + \dotsb + a_r < 2$.
\end{proof}
\begin{rem}\label{rem:prbundcontr}
  In the situation of Proposition~\ref{prop:prbundcontr}, the case $E \cong \sO_{\bP^1}^{\oplus (r+1)}$ means that $X \cong \bP^r \times \bP^1$, while the case $E \cong \sO_{\bP^1}^{\oplus r}\oplus \sO_{\bP^1}(1)$ means that $X$ is isomorphic to a blow-up of~$\bP^{r+1}$ along a codimension-$2$ linear subspace.
\end{rem}

The following Lemma will later be applied to the central fiber of the family:
\begin{lem}\label{lem:equidim}
  Let $X$ be a smooth projective variety and assume that $X$ is homeomorphic to a $\bP^r$-bundle over~$\bP^m$. Then any surjective morphism $\phi\colon X \to \bP^m$ is equidimensional (i.e., every fiber of~$\phi$ has dimension~$r$).
\end{lem}
\begin{proof}
  We mostly copy the proof for a slightly less general statement from~\cite{PS14}.

  Since $\dim X = m + r$, a general fiber of~$\phi$ must have dimension~$r$.
  Let $F_0$ be an irreducible component of a fiber of~$\phi$. Then $F_0$ gives rise to a class
  \[
    [F_0] \in H^{2k}(X, \bQ),
  \]
  where $k$ is the codimension of~$F_0$ in~$X$. Semicontinuity of fiber dimension yields $k \le m$. We must show that $k = m$.

  We first observe that by the Leray--Hirsch theorem, we have
  \[
    h^{2k}(X, \bQ) = \sum_{\ell = 0}^{k} h^{2k-2\ell}(\bP^r,\bQ)\cdot h^{2\ell}(\bP^m, \bQ) = \min\{k,r\} + 1,
  \]
  where we used that $h^{2\ell}(\bP^m, \bQ) = 1$ since $k \le m$. We let $d := \min\{k,r\}$ and denote by $L$ the class of an ample divisor on~$X$ and by~$H$ the class of a hyperplane in~$\bP^m$. We claim that the classes
  \begin{equation}
    L^d.(\phi^*H)^{k-d}, L^{d-1}.(\phi^*H)^{k-d+1}, \dotsc, L.(\phi^*H)^{k-1}, (\phi^*H)^k\label{eq:basis}
  \end{equation}
  form a basis of~$H^{2k}(X, \bQ)$, which can be seen as follows: By the dimension count established above, it suffices to show linear independency, so assume we are given $\lambda_0$, $\dotsc$, $\lambda_d \in \bQ$ such that
  \begin{equation}
    \sum_{\ell=0}^d \lambda_\ell L^{d-\ell}.(\phi^*H)^{k-d+\ell} = 0.\label{eq:2}
  \end{equation}
  Let $\ell_0 \in \{0, \dotsc, d\}$ and assume by induction that $\lambda_\ell = 0$ for all $\ell < \ell_0$. Then intersecting~\eqref{eq:2} with~$L^{r-d+\ell_0}.(\phi^*H)^{m-k+d-\ell_0}$ yields
  \[
    \lambda_{\ell_0} L^r.(\phi^*H)^m = 0,
  \]
  thus $\lambda_{\ell_0} = 0$ since $L^r.(\phi^*H)^m > 0$.

  So \eqref{eq:basis} is indeed a basis of~$H^{2k}(X, \bQ)$ and we can write
  \begin{equation}
    [F_0] = \sum_{\ell=0}^d \alpha_\ell L^{d-\ell}.(\phi^*H)^{k-d+\ell} \label{eq:5}
  \end{equation}
  for some $\alpha_0$, $\dotsc$, $\alpha_d \in \bQ$. We now again let $\ell_0 \in \{0, \dotsc, d\}$ and assume by induction that $\alpha_\ell = 0$ for all $\ell < \ell_0$. Then intersecting~\eqref{eq:5} with~$L^{r-d+\ell_0}.(\phi^*H)^{m-k+d-\ell_0}$ yields
  \begin{equation}
    L^{r-d+\ell_0}.(\phi^*H)^{m-k+d-\ell_0}.[F_0] = \alpha_{\ell_0} L^r.(\phi^*H)^m.\label{eq:6}
  \end{equation}
  For $\ell_0 < d$, the intersection product $(\phi^*H)^{m-k+d-\ell_0}.[F_0]$ vanishes, since $k \le m$ and $F_0$ maps to a point via~$\phi$. So we obtain $\alpha_0 = \dotsb = \alpha_{d-1} = 0$. Since $X$ is projective, this implies $\alpha_d \ne 0$.
  
  For $\ell_0 = d$, equation~\eqref{eq:6} yields
  \[
    L^r.(\phi^*H)^{m-k}.[F_0] = \alpha_dL^r.(\phi^*H)^m.
  \]
  As already observed, the right-hand side of this equation must be non-zero, while the left-hand side is non-zero if and only if $k = m$.
\end{proof}

We can now prove the main result of this section. Note that we must assume the family to be projective in order to apply the relative MMP.
\begin{thm}\label{thm:prbundp1}
  Let $p\colon \sX \to \Delta$ be a smooth projective morphism such that $X_{t} :=  p^{-1}(t)$ is a $\bP^r$-bundle over~$\bP^1$ for all $t \ne 0$. Then, after possibly shrinking $\Delta$, only the following two cases can occur:
  \begin{enumerate}[(i)]
  \item
    There exists a rank-$(r+1)$ vector bundle $\sE$ over $\bP^1 \times \Delta$ such that $\sX$ is isomorphic to~$\bP(\sE)$ over~$\Delta$.
  \item
    There exists a rank-$2$ bundle $\sE$ over~$\bP^r \times \Delta$ such that $\sX$ is isomorphic to~$\bP(\sE)$ over~$\Delta$ (this case can only occur if $X_t \cong \bP^r \times \bP^1$ for $t \ne 0$).
  \end{enumerate}
\end{thm}
\begin{proof}
  Since $X_{t}$ carries a $\bP^r$-bundle structure for $t \ne 0$, $K_{\sX}$ is not $p$-nef. By the relative cone and contraction theorems (cf.~\cite[Thm.~4.12]{Nak87}), there exists a relative Mori contraction $\Phi\colon \sX \to \sY$ over~$\Delta$ (after possibly shrinking $\Delta$), where $q\colon \sY \to \Delta$ is a projective morphism. Let $\phi_t := \Phi|_{X_t}\colon X_t \to Y_t := q^{-1}(t)$. Then $\phi_t$ is a Mori contraction for any~$t$. Since $\sY$ is irreducible, semicontinuity of fiber dimension implies that $q$ is equidimensional.

  In view of Proposition~\ref{prop:prbundcontr} and Remark~\ref{rem:prbundcontr}, the general fiber of~$q$ is isomorphic to~$\bP^m$ for some $m \in \{1, r, r+1\}$. We observe that, for any $t \in \Delta$, the restriction $\Pic \sX \to \Pic X_t$ is an isomorphism. In particular, there exists a line bundle $\sL \in \Pic \sX$ such that $\sL|_{X_t} \cong \phi_t^*\sO_{\bP^m}(1)$ for $t \ne 0$. It follows that $\sL$ is numerically trivial on the fibers of $\phi_0$, so, since $\phi_0$ is a Mori contraction,
  \[
    \sL|_{X_0} = \phi_0^* L'\quad\text{for some $L' \in \Pic Y_0$.}
  \]
  By semicontinuity, $h^0(L') \ge m+1$. Since furthermore $c_1(L')^n = 1$, we have $(Y_0, L') \cong (\bP^m, \sO_{\bP^m}(1))$ by~\cite[Thm~1.1]{KO73} (see also~\cite[I.1.2]{Fuj90a}), so in particular,
  \[
    \sY \cong \bP^m \times \Delta.
  \]

  We now first consider the case $m \in \{1, r\}$. This means that $\phi_t$ gives to~$X_t$ a $\bP^{r+1-m}$-bundle structure over~$\bP^m$ for $t \ne 0$. Since we have shown above that $Y_0 \cong \bP^m$, we can apply Lemma~\ref{lem:equidim} to conclude that $\phi_0$ is equidimensional. A general fiber of~$\phi_0$ is isomorphic to~$\bP^{r+1-m}$, since it is a smooth degeneration of the fibers of $\phi_t$, $t \ne 0$. We can now apply \cite[2.12]{Fuj87} to conclude that $\phi_0$ is indeed a $\bP^{r+1-m}$-bundle. Using again the fact that $\Pic \sX \to \Pic X_t$ is an isomorphism, we obtain a line bundle $\sH \in \Pic \sX$ whose restriction to any fiber of~$\Phi$ is isomorphic to~$\sO_{\bP^{r+1-m}}(1)$. Setting $\sE := \Phi_* \sH$, we obtain an isomorphism $\sX \cong \bP(\sE)$ over~$\Delta$.

  It remains to treat the case $m = r+1$. Then
  \[
    \Phi\colon \sX \to \sY = \bP^{r+1} \times \Delta
  \]
  is birational. Let $\sW \subset \sX$ be the exceptional set of~$\Phi$. Since $\Phi$ is a Mori contraction and $W_t := X_t \cap \sW$ is a divisor in $X_t$ for $t \ne 0$, it follows that $\sW$ is an irreducible divisor in~$\sX$. Let $\sZ := \Phi(\sW) \subset \sY$ and $Z_t := \sZ \cap Y_t$. Then $\sZ$ is an irreducible subscheme of $\sY \cong \bP^{r-1} \times \Delta$ and hence flat over the smooth curve~$\Delta$ (\cite[Prop.~III.9.7]{Har77}). Now by Proposition~\ref{prop:prbundcontr}, we know that $Z_t \subset \bP^{r+1}$ is a codimension-$2$ linear subspace for $t \ne 0$, so by flatness, $Z_0 \subset \bP^{r-1}$ is an $(r-1)$-dimensional subscheme of degree~$1$, hence $Z_0$ must also be a linearly embedded~$\bP^{r-1} \subset \bP^{r+1}$. It follows that $X_0$ is isomorphic to the blow-up of $Y_0 \cong \bP^{r+1}$ along $Z_0$ by~\cite[Theorem~1.1]{ESB89}. So in this case, $\sX$ is isomorphic to $\bP(\sO_{\bP^1}^{\oplus r} \oplus \sO_{\bP^1}(1)) \times \Delta$ over~$\Delta$.
\end{proof}
\begin{rem}
  $\bP^r$-bundles over~$\bP^1$ were already studied for arbitrary~$r$ by Brieskorn in~\cite{Bri65}. The problem of classifying (global) deformations of such bundles was raised there, but was only solved for~$r=1$. The case $r=2$ has been considered in~\cite{Nak98}.
\end{rem}

\section{Almost homogeneous projective bundles}\label{sec:almhomprojbun}
In this section, we investigate criteria for projective bundles to be almost homogeneous.

The following Lemma enables us to construct from a given almost homogeneous projective bundle an almost homogeneous projective bundle of higher dimension:
\begin{lem}\label{lem:sumlinebund}
  Let $M$ be a complex manifold with $H^1(\sO_M) = 0$ and let $E$ be a vector bundle on~$M$ such that $\bP(E)$ is almost homogeneous. If $H^0(E) \ne 0$ and $H^1(E^*) = 0$, then also $\bP(E \oplus \sO_M)$ is almost homogeneous.
\end{lem}
\begin{proof}
  Let $0 \ne s \in H^0(E)$. For any $t \in \bC$, we consider the section $\tilde{s}_t := (ts, 1) \in H^0(E \oplus \sO_M)$. We obtain a short exact sequence of vector bundles
  \[
    0 \longrightarrow \sO_M \stackrel{\cdot\tilde{s}_t}{\longrightarrow} E \oplus \sO_M \longrightarrow Q_t \longrightarrow 0.
  \]
  It follows from the snake lemma that $Q_t \cong E$ for any $t \in \bC$.

  We now write $X := \bP(E \oplus \sO_M)$ and interpret $\tilde{s}_t$ as an element of $H^0(\sO_X(1))$, where we denote by $\sO_X(1)$ the tautological line bundle on $\bP(E \oplus \sO_M)$. The zero locus of this section is a divisor $Y_t \subset X$ with $Y_t \cong \bP(Q_t) \cong \bP(E)$ for any $t \in \bC$. Its normal bundle is given by $N_{Y_t|X} \cong \sO_{Y_t}(1)$, so we have in particular $H^0(N_{Y_t|X}) \cong H^0(E)$ by the Leray spectral sequence.

  Clearly, the one-dimensional subpace of~$H^0(N_{Y_t|X})$ defined by the family~$(Y_t)_{t \in \bC}$ is generated by the section $s \in H^0(E)$ chosen at the beginning. If we consider the long exact cohomology sequence associated to the normal bundle sequence
  \begin{equation}\label{eq:normbund}
    0 \longrightarrow T_{Y_t} \longrightarrow T_X|_{Y_t} \longrightarrow N_{Y_t|X} \longrightarrow 0,
  \end{equation}
  the section $s \in H^0(N_{Y_t|X})$ maps to an element $\xi \in H^1(T_{Y_t})$ which describes the infinitesimal change of the complex structure on~$Y_t$ in the family~$(Y_t)_{t \in \bC}$. But we have seen above that all $Y_t$ are isomorphic, so it follows that $\xi = 0$. This implies that $s$ lifts to a section in $H^0(T_X|_{Y_t})$. Since $T_{Y_t} \cong T_{\bP(E)}$ is generically globally generated by hypothesis, the sequence~\eqref{eq:normbund} then implies that also $T_X|_{Y_t}$ is generically globally generated.

  We now let $x \in Y_t$ be a general point such that $T_{X,x}$ is generated by global sections in~$H^0(T_{X}|_{Y_t})$. In order to show that $T_{X,x}$ is also generated by global sections in~$H^0(T_X)$, it is sufficient to prove that the natural restriction map $H^0(T_X) \to H^0(T_{X}|_{Y_t})$ is surjective. By the long exact cohomology sequence associated to
  \[
    0 \longrightarrow T_X(-Y_t) \longrightarrow T_X \longrightarrow T_X|_{Y_t} \longrightarrow 0,
  \]
  it suffices to show that $H^1(T_X(-Y_t)) = 0$. To this aim, we first denote by $\pi\colon X \to M$ the natural projection and consider the relative tangent sequence tensorized by~$\sO_X(-Y_t)$:
  \begin{equation}\label{eq:reltang}
    0 \longrightarrow T_{X/M}(-Y_t) \longrightarrow T_{X}(-Y_t) \longrightarrow \pi^*T_M(-Y_t) \longrightarrow 0.
  \end{equation}
  Since $\sO_X(-Y_t) \cong \sO_X(-1)$, we have $H^q(\pi^*T_M(-Y_t)) = 0$ for all~$q$ by the Leray spectral sequence. From the long exact cohomology sequence associated to~\eqref{eq:reltang}, we thus conclude $H^1(T_X(-Y_t)) \cong H^1(T_{X/M}(-Y_t))$.

  We finally consider the relative Euler sequence
  \[
    0 \longrightarrow \sO_X \longrightarrow \sO_X(1) \otimes \pi^* (E^*\oplus \sO_M) \longrightarrow T_{X/M} \longrightarrow 0.
  \]
  Tensorizing this sequence by~$\sO_X(-Y_t) \cong \sO_X(-1)$ yields
  \[
    0 \longrightarrow \sO_X(-Y_t) \longrightarrow \pi^* (E^* \oplus \sO_M) \longrightarrow T_{X/M}(-Y_t) \longrightarrow 0.
  \]
  Since again $H^q(\sO_X(-Y_t)) = 0$ for all~$q$ by Leray, we obtain
  \[
    H^1(T_{X/M}(-Y_t)) \cong H^1(\pi^* (E^* \oplus \sO_M)) \cong H^1(E^* \oplus \sO_M) = 0.
  \]
  So we conclude that $H^1(T_X(-Y_t)) = 0$, which implies by our previous considerations that $T_{X,x}$ is generated by global sections in~$H^0(T_X)$.
\end{proof}

As a first application of this Lemma, we consider projective bundles over~$\bP^1$ and reprove a classically known fact (cf.~\cite{Bri65}):
\begin{cor}\label{cor:prbundp1}
  Let $X$ be a $\bP^r$-bundle over $\bP^1$. Then $X$ is almost homogeneous.
\end{cor}
\begin{proof}
  Any $\bP^r$-bundle over $\bP^1$ is of the form $\bP(V)$ for some rank-$(r+1)$-bundle $V$ on~$\bP^1$. By Grothendieck, $V$ splits as a direct sum of line bundles, so, after tensorizing with a suitable line bundle, we can assume
  \[
    V \cong \sO_{\bP^1} \oplus \sO_{\bP^1}(a_1) \oplus \dotsb \oplus \sO_{\bP^1}(a_r), \qquad a_1, \dotsc, a_r \ge 0.
  \]
  If we now let $E := \sO_{\bP^1}(a_1) \oplus \dotsb \oplus \sO_{\bP^1}(a_r)$, we can assume by induction on~$r$ that $\bP(E)$ is almost homogeneous. But now clearly $H^0(E) \ne 0$ and $H^1(E) = 0$, so we can apply Lemma~\ref{lem:sumlinebund} to conclude that $X = \bP(V)$ is almost homogeneous.
\end{proof}

In a similar spirit, the following Corollary constructs higher-dimensional almost homogeneous projective bundles over $\bP^n$ for $n \ge 2$ if  given an almost homogeneous projective bundle $\bP(E)$ over $\bP^n$.
An important application is the case $E = T_{\bP^n}$, the varieties $\bP(T_{\bP^n})$, $n \ge 2$, being classical examples of rational homogeneous spaces.
\begin{cor}\label{cor:rk2plus}
  Let $E$ be a vector bundle on~$\bP^n$, $n \ge 2$, such that $\bP(E)$ is almost homogeneous. Let $d_1$, $\dotsc$, $d_r \in \bZ$ such that $H^0(E(-d_j)) \ne 0$ and $H^1(E^*(d_j)) = 0$ for all $j = 1$, $\dotsc$, $r$. Then $\bP(E \oplus \sO_{\bP^n}(d_1) \oplus \dotsb \oplus \sO_{\bP^n}(d_r))$ is almost homogeneous. In particular, $\bP(T_{\bP^n}\oplus \sO(1)^{\oplus r})$ is almost homogeneous for any~$r\ge 0$.
\end{cor}
\begin{proof}
  The case $r=0$ is trivial; for $r \ge 1$ we let $E' := E \oplus \sO_{\bP^n}(d_1) \oplus \dotsb \oplus \sO_{\bP^n}(d_{r-1})$. By induction, we can assume that $\bP(E')$ is almost homogeneous. The given hypotheses imply $H^0(E'(-d_r)) \ne 0$ and $H^1(E'^*(d_r)) = 0$. So we can apply Lemma~\ref{lem:sumlinebund} to conclude that also $\bP(E'(-d_r) \oplus \sO_{\bP^n}) \cong \bP(E' \oplus \sO_{\bP^n}(d_r)) = \bP(E \oplus \sO_{\bP^n}(d_1) \oplus \dotsb \oplus \sO_{\bP^n}(d_r))$ is almost homogeneous.
\end{proof}

In order to deal with global deformations of almost homogeneous projective bundles, we need criteria for certain projectivized unstable vector bundles to be almost homogeneous:
\begin{lem}\label{lem:almhombun}
  Let $M$ be a compact Kähler manifold and let $E$ be a rank-$2$ vector bundle on~$M$ which is slope-unstable with respect to a Kähler form $\omega$ on~$M$. Let $D \subset E$ be the maximally destabilizing subsheaf and let $Z \subset M$ be the codimension-$2$ locally complete intersection subscheme defined by the natural short exact sequence
  \begin{equation}\label{eq:serre}
    0 \longrightarrow D \longrightarrow E \longrightarrow D^* \otimes \det E  \otimes \sI_Z \longrightarrow 0.
  \end{equation}

  \begin{enumerate}[(i)]
  \item\label{it:nec}
    Let $r \in \bN_0$ and $L_1$, $\dotsc$, $L_r \in \Pic M$.
    If $\bP(E \oplus L_1 \oplus \dotsb \oplus L_r)$ is almost homogeneous, then the group
    \[
      \Aut^0_Z M := \{\varphi\in\Aut^0 M \mid \varphi(Z) = Z\}
    \]
    acts on~$M$ with an open orbit.
  \item\label{it:suff}
    Suppose $H^1(\sO_M) = 0$, $H^1(D^{\otimes 2} \otimes \det E^*) = 0$ and $h^0(D^{\otimes 2} \otimes \det E^*) \ge 2$ (these conditions are automatically satisfied for $M = \bP^n$). If $\Aut^0_Z M$ acts on~$M$ with an open orbit, then $\bP(E)$ is almost homogeneous.
  \end{enumerate}
\end{lem}
\begin{proof}
  Observe first that the maximally destabilizing subsheaf $D \subset E$ is uniquely determined.

  In order to prove~(\ref{it:nec}), we let $\tilde{E} := E \oplus L_1 \oplus \dotsb \oplus L_r$ and suppose that $\bP(\tilde{E})$ is almost homogeneous. We denote by $\pi \colon \bP(\tilde{E}) \to M$ the natural projection. Since $\pi_* \sO_{\bP(\tilde{E})} = \sO_M$, any automorphism $\psi \in \Aut^0 \bP(\tilde{E})$ induces an automorphism $\varphi \in \Aut^0 M$ such that the diagram
  \begin{equation}\label{eq:fibaut}
    \xymatrix{
      \bP(\tilde{E}) \ar[d]^\pi \ar[r]^\psi &\bP(\tilde{E})\ar[d]^\pi\\
      M \ar[r]^\varphi &M}
  \end{equation}
  commutes. In the case that $E$ is a direct sum of two line bundles, we have $Z = \emptyset$, so by the existence of diagram~\eqref{eq:fibaut}, $\Aut^0_Z M = \Aut^0 M$ acts on~$M$ with an open orbit.

  In the case that $E$ is not a direct sum of two line bundles, we must study diagram~\eqref{eq:fibaut} more thoroughly. By the universal property of the pullback, $\psi$ factors as $\psi = \tau\circ\tilde{\psi}$ where $\tau$ and $\tilde{\psi}$ are such that the diagram
  \begin{equation}\label{eq:pullback}
    \xymatrix{
      \bP(\tilde{E}) \ar[dr]^\pi \ar[r]^{\tilde{\psi}} &\bP(\varphi^* \tilde{E}) \ar[r]^\tau \ar[d]^{\varphi^*\pi} &\bP(\tilde{E}) \ar[d]^\pi\\
      &M \ar[r]^\varphi &M}
  \end{equation}
  commutes. Since $\psi$ and $\tau$ are isomorphisms, also $\tilde{\psi}$ must be an isomorphism. This implies that there exists a line bundle $L \in \Pic M$ such that $\varphi^* \tilde{E} \cong \tilde{E} \otimes L$, i.e.,
  \[
    \varphi^*E \oplus \varphi^*L_1 \oplus \dotsb \oplus \varphi^*L_r \cong (E\otimes L) \oplus (L_1\otimes L) \oplus \dotsb \oplus (L_r\otimes L).
  \]
  Since $E$ is unsplit, we must have
  \begin{equation}\label{eq:samebund}
    \varphi^* E \cong E \otimes L
  \end{equation}
  by the uniqueness of the direct sum decomposition of vector bundles. Since $\varphi^*$ acts trivially on~$H^2(M, \mathbb{R})$, the Kähler form $\varphi^* \omega$ is numerically equivalent to~$\omega$. This implies that $\varphi^* D$ is the maximally $\omega$-destabilizing subsheaf of~$\varphi^* E$. Furthermore, clearly $D \otimes L$ is the maximally ($\omega$-)destablilizing subsheaf of~$E \otimes L$. By~\eqref{eq:samebund} and the uniqueness of the maximally destabilizing subsheaf, the two short exact sequences one obtains from~\eqref{eq:serre} by applying $\varphi^*$ respectively by tensorizing with~$L$ must be the same, so in particular, we obtain $\sI_{\varphi^{-1}(Z)} \cong \sI_Z$, so $\varphi(Z) = Z$, finishing the proof of~(\ref{it:nec}).

  Turning to~(\ref{it:suff}), we now suppose that $\Aut^0_ZM$ acts on~$M$ with an open orbit. Let $\varphi \in \Aut^0 M$ be any automorphism with $\varphi(Z) = Z$. Applying $\varphi^*$ to~\eqref{eq:serre}, we obtain
  \[
    0 \longrightarrow \varphi^* D \longrightarrow \varphi^* E \longrightarrow \varphi^* D^* \otimes \det \varphi^*E  \otimes \sI_Z \longrightarrow 0.
  \]
  Since $H^1(\sO_M) = 0$, we have $\varphi^* D \cong D$ and $\det \varphi^* E \cong \det E$, so the bunde $\varphi^* E$ is an extension of the same rank-$1$ sheaves as the bundle $E$. Since $H^1(D^{\otimes 2} \otimes \det E^*) = 0$, there is, up to isomorphism, only one locally free sheaf given by such an extension (cf.~\cite[§1.5.1]{OSS80}), so we obtain $\varphi^* E \cong E$. In particular, we get an isomorphism $\tilde{\psi}\colon \bP(E) \to \bP(\varphi^* E)$ over~$M$ such that the diagram
  \[
    \xymatrix{
      \bP({E}) \ar[dr]^\pi \ar[r]^{\tilde{\psi}} &\bP(\varphi^* {E}) \ar[r]^\tau \ar[d]^{\varphi^*\pi} &\bP(E) \ar[d]^\pi\\
      &M \ar[r]^\varphi &M}
  \]
  is commutative, where the right-hand part of this diagram is just the pull-back square analogous to diagram~\eqref{eq:pullback}. If we let $\psi := \tau \circ \tilde{\psi}$, we get an automorphism $\psi \in \Aut \bP(E)$ which induces the given automorphism $\varphi \in \Aut^0_Z M$.

  To conclude, it suffices to show that there exists an at least $1$-dimensional family of automorphisms of~$\bP(E)$ over~$M$. This is equivalent to showing that
  \[
    h^0(\End E) = h^0(E^* \otimes E) \ge 2.
  \]
  But now, tensoring \eqref{eq:serre} with $E^* \cong E \otimes \det E^*$ yields
  \[
    0 \longrightarrow E\otimes D \otimes \det E^* \longrightarrow E^*\otimes E \longrightarrow E \otimes D^* \otimes \sI_Z \longrightarrow 0.
  \]
  Finally, $D \subset E$ implies $D^{\otimes 2} \otimes \det E^* \subset E \otimes D \otimes \det E^*$, so we obtain
  \[
    h^0(E^* \otimes E) \ge h^0(E \otimes D \otimes \det E^*) \ge h^0(D^{\otimes 2} \otimes \det E^*) \ge 2.\qedhere
  \]
\end{proof}

\section{Examples of global deformations}\label{sec:degen}
In order to obtain global deformations of projectivized rank-$2$ vector bundles where almost homogeneity is not preserved, we cite the following construction by Strømme:
\begin{prop}[{\cite[§4]{Str83}}]\label{prop:strom}
  Let $E$ be a rank-$2$ vector bundle on~$\bP^2$ with first Chern class $c_1 := c_1(E)$. Let $d \in \bZ$ and suppose that we are given
  \begin{enumerate}[(i)]
  \item
    a section $\tau \in H^0(E(d-c_1))$ vanishing along a subscheme $Z \subset \bP^2$ of codimension~$2$,
  \item
    a section $F \in H^0(\sO_{\bP^2}(2d-c_1))$ whose zero divisor is disjoint from~$Z$.
  \end{enumerate}
  Then there exists a rank-$2$ vector bundle $\mathcal{E}$ on $\bP^2 \times \Delta$ such that $E_t := \mathcal{E}|_{\bP^2 \times \{t\}} \cong E$ for $t \ne 0$ and $E_0 \cong D$, where $D$ is a rank-$2$ vector bundle on~$\bP^2$ fitting in the short exact sequence
  \[
    0 \longrightarrow \sO_{\bP^2}(d) \longrightarrow D \longrightarrow \sI_Z(c_1-d) \longrightarrow 0.
  \]
  We have $D \cong E$ if and only if $H^0(E(-d)) \ne 0$.
\end{prop}
\begin{rem}
  The assumptions in Proposition~\ref{prop:strom} can be satisfied for any rank-$2$ bundle $E$ on $\bP^2$ provided $d$ is chosen sufficiently large.
\end{rem}

\begin{cor}\label{cor:stromtang}
  For any $d \ge 2$ there exists a rank-$2$ vector bundle $\sE$ on $\mathbb{P}^2\times\Delta$ with the following properties:
\begin{enumerate}[(i)]
\item\label{first}
  $E_t := \mathcal{E}|_{\mathbb{P}^2\times\{t\}}$ is isomorphic to $T_{\mathbb{P}^2}$ for all $t \ne 0$,
\item\label{second}
  $E_0$ sits inside a short exact sequence
  \begin{equation}\label{eq:degsequence}
    0 \longrightarrow \sO_{\bP^2}(d) \longrightarrow E_0 \longrightarrow \sI_Z(3-d) \longrightarrow 0,
  \end{equation}
  where $Z \subset \bP^2$ is a codimension-$2$ locally complete intersection given as the zero locus of a section in $H^0(T_{\bP^2}(d-3))$.
\end{enumerate}
\end{cor}
\begin{proof}
  Since $d \ge 2$, the vector bundle $T_{\bP^2}(d-3)$ is globally generated, so a general section in $H^0(T_{\bP^2}(d-3))$ vanishes along a subscheme $Z \subset \bP^2$ of codimension $2$. Furthermore, $2d - 3 \ge 1$, so there exists a divisor in $\lvert \sO_{\bP^2}(2d - 3) \rvert$ disjoint from~$Z$. Thus we can apply Proposition~\ref{prop:strom}.
\end{proof}

We can use this construction to get examples for global deformations of almost homogeneous projective bundles:
\begin{thm}\label{thm:notalmhom}
  Fix $d \ge 2$. Let $\sE$ be the rank-$2$ bundle on $\bP^2 \times \Delta$ constructed in Corollary~\ref{cor:stromtang}. For any $r \ge 0$ we consider the family
  \[
    \sX := \bP(\sE \oplus \sO_{\bP^2\times\Delta}(1)^{\oplus r}) \stackrel{\pi}{\longrightarrow} \bP^2 \times \Delta
  \]
  over~$\Delta$. Then:
  \begin{enumerate}[(i)]
  \item\label{it:almhom}
    For any $t \ne 0$, the fiber $X_t \cong \bP(T_{\bP^2} \oplus \sO_{\bP^2}(1)^{\oplus r})$ is an $(r+3)$-dimensional almost homogeneous manifold.
  \item\label{it:centr}
    If $d \ge 4$, the fiber $X_0 \cong \bP(E_0 \oplus \sO_{\bP^2}(1)^{\oplus r})$ is not almost homogeneous.
  \item\label{it:notfano}
    $X_t$ is Fano if and only if $t \ne 0$.
  \end{enumerate}
\end{thm}

\begin{proof}
  (\ref{it:almhom}) is just Corollary~\ref{cor:rk2plus}.
  
  For~(\ref{it:centr}), we note that by Lemma~\ref{lem:almhombun}~(\ref{it:nec}), it suffices to show that $\{\varphi \in \Aut^0 \bP^2 \mid \varphi(Z) = Z \}$ does not act on~$\bP^2$ with an open orbit. Since $Z$ is given as the zero locus of a section $\tau \in H^0(T_{\bP^2}(d-3))$, there is an exact sequence
  \[
    0 \to \sO_{\bP^2} \longrightarrow T_{\bP^2}(d-3) \longrightarrow \sI_Z(2d-3) \longrightarrow 0.
  \]
  Tensorizing with~$T_{\bP^2}(3-2d)$ yields
  \[
    0 \to T_{\bP^2}(3-2d) \longrightarrow T_{\bP^2}\otimes T_{\bP^2}(-d) \longrightarrow T_{\bP^2}\otimes\sI_Z \longrightarrow 0.
  \]
  But now since $d \ge 4$, we have
  \[
    h^1(T_{\bP^2}(3-2d)) = h^1(\Omega^1_{\bP^2}(2d-6)) = 0,
  \]
  and from the Euler sequence it follows that also $H^0(T_{\bP^2} \otimes T_{\bP^2}(-d)) = 0$. So we obtain $H^0(T_{\bP^2}\otimes\sI_Z) = 0$ which means that there are no non-zero vector fields on $\bP^2$ vanishing along~$Z$.

  It remains to prove~(\ref{it:notfano}). We denote by $\pi_t\colon X_t \to \bP^2$ the restriction of $\pi$ to~$X_t$. Standard calculations yield
  \[
    K_{X_t}^* = \sO_{X_t}(r+2) \otimes \pi_t^* \sO_{\bP^2}(-r).
  \]
  Thus $X_t$ is Fano if and only if the $\bQ$-line bundle
  \begin{equation}\label{eq:qlinbund}
    \sO_{X_t}(1) \otimes \pi_t^* \sO_{\bP^2}\bigl( -\tfrac{r}{r+2} \bigr)
  \end{equation}
  is ample. But if we now restrict the short exact sequence~\eqref{eq:degsequence} to lines in~$\bP^2$, it follows immediately that $E_0$ is not ample. So also $E_0 \oplus \sO_{\bP^2}(1)^{\oplus r}$ is not ample, which is equivalent to the statement that $\sO_{X_0}(1)$ is not ample. Thus also \eqref{eq:qlinbund} is not ample for $t=0$. For $t \ne 0$, we observe that $T_{\bP^2}(q)$ is an ample $\bQ$-bundle for all rational $q > -1$, thus \eqref{eq:qlinbund} is ample for $t \ne 0$.
\end{proof}
The argument used to prove Theorem~\ref{thm:notalmhom}~(\ref{it:notfano}) can be used to show that nontrivial global deformations of $\bP(T_{\bP^2})$ cannot be Fano. Furthermore, the the assumption $d \ge 4$ in Theorem~\ref{thm:notalmhom} (\ref{it:centr}) is actually necessary:
\begin{thm}\label{thm:fanoalmhom}
  Let $\mathcal{E}$ be a rank-$2$ vector bundle on $\bP^2 \times \Delta$ such that $E_t := \mathcal{E}|_{\bP^2 \times \{t\}}$ is isomorphic to $T_{\bP^2}$ for all $t \ne 0$.
  \begin{enumerate}[(i)]
  \item
    Assume that $\bP(E_0)$ is Fano. Then $E_0 \cong T_{\bP^2}$.
  \item
     Suppose that $H^0(E_0(-4)) = 0$. Then $\bP(E_0)$ is almost homogeneous.
  \end{enumerate}
\end{thm}
\begin{proof}
  If we let
  \[
    d:= \max \{ \ell \in \bZ \mid H^0(E_0(-\ell)) \ne 0 \},
  \]
  we have an exact sequence
  \[
    0 \longrightarrow \sO_{\bP^n}(d) \longrightarrow E_0 \longrightarrow \sI_Z(3-d) \longrightarrow 0,
  \]
  where $Z \subset \bP^2$ is a $0$-dimensional locally complete intersection subscheme of length
  \[
    \ell(Z) = c_2(E_0(-d)) = c_2(T_{\bP^2}(-d)) = 3 - 3d + d^2.
  \]

  By assumption, we have $d < 4$, and by semicontinuity, $d \ge 1$. If $d = 1$, it is a classic result that $E_0 \cong T_{\bP^2}$ (see for example \cite[§2.3.2]{OSS80}), so we can assume that $d \ge 2$. In this case, the argument from the proof of Theorem~\ref{thm:notalmhom}~(\ref{it:notfano}) applies verbatim to show that $\bP(E_0)$ cannot be Fano. Furthermore, we have $2d > 3 = c_1(E_0)$, so $E_0$ is an unstable bundle. Now if $d = 2$, we have $\ell(Z) = 1$, and if $d = 3$, we get $\ell(Z) = 3$. It is easily verified that for any $0$-dimensional subscheme $Z \subset \bP^2$ with $\ell(Z) \le 3$, the group $\{\varphi\in\Aut^0(\bP^2) \mid \varphi(Z) = Z\}$ acts on~$\bP^2$ with an open orbit. The claim that $\bP(E_0)$ is almost homogeneous now follows from Lemma~\ref{lem:almhombun} (\ref{it:suff}).
\end{proof}
\begin{rem}
  A construction analogous to the one in Corollary~\ref{cor:stromtang} was given in~\cite[§6]{PS14} for $T_{\bP^n}$ for arbitrary~$n \ge 2$. Unfortunately, for $n \ge 3$, the arguments used in the proofs of Theorems \ref{thm:notalmhom} and \ref{thm:fanoalmhom} no longer work, so it is unclear whether one can obtain non-almost homogeneous global deformations of~$\bP(T_{\bP^n})$ for $n \ge 3$.
\end{rem}

\bibliographystyle{amsalpha}
\bibliography{bibtex/references}

\providecommand{\bysame}{\leavevmode\hbox to3em{\hrulefill}\thinspace}
\providecommand{\MR}{\relax\ifhmode\unskip\space\fi MR }
\providecommand{\MRhref}[2]{%
  \href{http://www.ams.org/mathscinet-getitem?mr=#1}{#2}
}
\providecommand{\href}[2]{#2}
\begin{thebibliography}{{Mok}16}

\bibitem[Bot57]{Bot57}
Raoul Bott, \emph{Homogeneous vector bundles}, Ann. of Math. (2) \textbf{66}
  (1957), 203--248.

\bibitem[Bri65]{Bri65}
Egbert Brieskorn, \emph{\"{U}ber holomorphe {$P_{n}$}-{B}{\"u}ndel {\"u}ber
  {$P_{1}$}}, Math. Ann. \textbf{157} (1965), 343--357.

\bibitem[Cat04]{Cat04}
Fabrizio Catanese, \emph{Deformation in the large of some complex manifolds.
  {I}}, Ann. Mat. Pura Appl. (4) \textbf{183} (2004), no.~3, 261--289.

\bibitem[ESB89]{ESB89}
Lawrence Ein and Nicholas Shepherd-Barron, \emph{Some special {C}remona
  transformations}, Amer. J. Math. \textbf{111} (1989), no.~5, 783--800.

\bibitem[Fuj87]{Fuj87}
Takao Fujita, \emph{On polarized manifolds whose adjoint bundles are not
  semipositive}, Algebraic geometry, {S}endai, 1985, Adv. Stud. Pure Math.,
  vol.~10, North-Holland, Amsterdam, 1987, pp.~167--178.

\bibitem[Fuj90]{Fuj90a}
\bysame, \emph{Classification theories of polarized varieties}, London
  Mathematical Society Lecture Note Series, vol. 155, Cambridge University
  Press, Cambridge, 1990.

\bibitem[Har77]{Har77}
Robin Hartshorne, \emph{Algebraic geometry}, Springer-Verlag, New
  York-Heidelberg, 1977, Graduate Texts in Mathematics, No. 52.

\bibitem[KO73]{KO73}
Shoshichi Kobayashi and Takushiro Ochiai, \emph{Characterizations of complex
  projective spaces and hyperquadrics}, J. Math. Kyoto Univ. \textbf{13}
  (1973), 31--47.

\bibitem[Kod66]{Kod66}
K.~Kodaira, \emph{On the structure of compact complex analytic surfaces. {II}},
  Amer. J. Math. \textbf{88} (1966), 682--721.

\bibitem[KS58]{KS58}
K.~Kodaira and D.~C. Spencer, \emph{On deformations of complex analytic
  structures. {I}, {II}}, Ann. of Math. (2) \textbf{67} (1958), 328--466.

\bibitem[{Mok}16]{Mok16}
Ngaiming {Mok}, \emph{{Geometric structures and substructures on uniruled
  projective manifolds.}}, {Foliation theory in algebraic geometry. Proceedings
  of the conference, New York, NY, USA, September 3--7, 2013}, Cham: Springer,
  2016, pp.~103--148 (English).

\bibitem[Nak87]{Nak87}
Noboru Nakayama, \emph{The lower semicontinuity of the plurigenera of complex
  varieties}, Algebraic geometry, {S}endai, 1985, Adv. Stud. Pure Math.,
  no.~10, North-Holland, Amsterdam, 1987, pp.~551--590.

\bibitem[Nak98]{Nak98}
Iku Nakamura, \emph{Global deformations of {${\bf P}^2$}-bundles over {${\bf
  P}^1$}}, J. Math. Kyoto Univ. \textbf{38} (1998), no.~1, 29--54.

\bibitem[OSS80]{OSS80}
Christian Okonek, Michael Schneider, and Heinz Spindler, \emph{Vector bundles
  on complex projective spaces}, Progress in Mathematics, no.~3, Birkh\"auser
  Boston, Mass., 1980.

\bibitem[Pot69]{Pot69}
Joseph Potters, \emph{On almost homogeneous compact complex analytic surfaces},
  Invent. Math. \textbf{8} (1969), 244--266.

\bibitem[PP10]{PP10}
Boris Pasquier and Nicolas Perrin, \emph{Local rigidity of quasi-regular
  varieties}, Math. Z. \textbf{265} (2010), no.~3, 589--600.

\bibitem[PS14]{PS14}
Thomas Peternell and Florian Schrack, \emph{Contact {K}\"ahler manifolds:
  symmetries and deformations}, Algebraic and complex geometry, Springer Proc.
  Math. Stat., vol.~71, Springer, Cham, 2014, pp.~285--308.

\bibitem[Str83]{Str83}
Stein~Arild Str{\o}mme, \emph{Deforming vector bundles on the projective
  plane}, Math. Ann. \textbf{263} (1983), no.~3, 385--397.

\end{thebibliography}
\end{document}